\documentclass[12pt,a4paper]{amsart} 
\usepackage{amsmath}
\usepackage{amsthm}
\usepackage{amssymb}
\usepackage{ot1patch}
\usepackage{mathrsfs}

\newtheorem{thm}{Theorem}[section]
\newtheorem{lem}[thm]{Lemma}
\newtheorem{cor}[thm]{Corollary}
\newtheorem{prop}[thm]{Proposition}

\theoremstyle{remark}
\newtheorem{rem}[thm]{Remark}
\newtheorem*{rem*}{Remark}

\theoremstyle{definition}
\newtheorem{dfn}[thm]{Definition}
\newtheorem{ex}[thm]{Example}

\numberwithin{equation}{section}

\newcommand{\C}{\mathbb{C}}
\newcommand{\Rz}{\mathbb{R}}

\begin{document}

\title{On the complex \L ojasiewicz inequality with parameter}

\author{Maciej P. Denkowski
}\address{Jagiellonian University, Faculty of Mathematics and Computer Science, Institute of Mathematics, \L ojasiewicza 6, 30-348 Krak\'ow, Poland}\email{maciej.denkowski@uj.edu.pl}\date{April 15th 2014, \textit{revised:} June 6th 2014}

\begin{abstract} 
We prove a continuity property in the sense of currents of a continuous family of holomorphic functions which allows us to obtain a \L ojasiewicz inequality with an effective exponent independent of the parameter.
\end{abstract}

\maketitle

\section{Introduction}

The \textit{\L ojasiewicz inequality} introduced in \cite{L} is one of the most important tools in singularity theory, both complex and real. The first result concerning a parametrized family --- but, of course, with \textit{an exponent that is independent of the parameter} --- is due to \L ojasiewicz and Wachta \cite{LW}. Fairly recently, we have obtained in \cite{D4} an effective \L ojasiewicz inequality with parameter in complex analytic geometry, using only complex analytic methods. This article is somehow a continuation of that work, inspired to some extent by the observations made in \cite{D3} and the intersection theory results introduced in \cite{T}.

Our best results are presented in the following theorem. Throughout the paper we assume that the topological space $T$ is \textit{1st countable}.
\begin{thm}
Assume that $f\colon T\times\Omega\to{\C}$ is a continuous function where $T$ is a locally compact, connected topological space,  $\Omega\subset{\C}^m$ is a domain, and for all $t\in T$, $f_t\in\mathcal{O}(\Omega)$ does not vanish identically. Assume moreover that $0\in\Omega$ and $f_{t}(0)=0$ for any $t$. Then \begin{enumerate}
\item $Z_{f_t}\to Z_{f_0}$ in the sense of currents, where $Z_{f_t}$ denotes the cycle of zeroes of $f_{t}$;
\item there is a neighbourhood $U\subset\Omega$ of zero in which, for all $t$ close enough to $t_0$, $$|f_t(x)|\geq c(t)\mathrm{dist}(x,f_t^{-1}(0))^{\alpha},$$ where $c(t)>0$ is a constant depending on the parameter, but the exponent $\alpha=\mathrm{ord}_0 f_0$ is uniform.
\end{enumerate}
\end{thm}

For the convenience of the reader let us recall two basic notions of convergence of sets, especially useful in analytic geometry (see e.g. \cite{DD} and \cite{TW}). We consider the following situation: $T$ is  a topological space and $E\subset T\times{\Rz}^n$ is a set with closed sections $E_t=\{x\in{\Rz}^n\mid (t,x)\in E\}$ and we put $F:=\pi(E)$ for $\pi(t,x)=t$. Assume that $t_0$ is an accumulation point of $F$.

\begin{dfn}(see e.g. \cite{DD})
We say that $E_t$ \textit{converges in the sense of Kuratowski} to a set $A$, when $t\to t_0$, if \begin{itemize}
\item for any $x\in A$, for any neighbourhood $U$ of $x$, there is a neighbourhood $V$ of $t_0$ such that $U\cap E_t\neq\varnothing$ for all $t\in V\cap F\setminus \{t_0\}$;
\item if $x$  is such that for any neighbourhood $U\ni x$ and any neighbourhood $V\ni t_0$ there is a point $t\in V\setminus \{t_0\}$ such that $U\cap E_t\neq\varnothing$, then $x\in A$.
\end{itemize}
We write then $E_t\stackrel{K}{\longrightarrow} A$.\\ If for each $t_0$, $E_t\stackrel{K}{\longrightarrow} E_{t_0}$, then we say that $E$ \textit{has continuously varying fibres}.
\end{dfn}
\begin{rem}\label{zb}
It is easy to see (cf. \cite{TW}, \cite{DD}) that this convergence for the graphs of a sequence continuous functions is precisely the \textit{local uniform convergence} of the functions themselves.
\end{rem}

We have the following straightforward observation:
\begin{lem}
If any point in $T$ has a countable basis of neighbourhoods, then $E_t\stackrel{K}{\longrightarrow} A$ when $t\to t_0$ iff \begin{itemize}
 \item if $x\in A$, then for any sequence $t_\nu\to t_0$ we can find points $E_{t_\nu}\ni x_\nu\to x$;
 \item if $x$ is such that there is a sequence $t_\nu\to t_0$ and points $E_{t_\nu}\ni x_\nu\to x$, then $x\in A$.
 \end{itemize}
\end{lem}

In complex analytic geometry this kind of convergence is very useful for different purposes (Bishop's Theorem, algebraic approximation as in \cite{B} or algebraicity criteria as in \cite{DP}). We may refine it taking into account multiplicities (cf. \cite{T} and \cite{Ch}). In order to do so, consider a sequence of \textit{positive pure $k$-dimensional analytic cycles} (\footnote{A positive pure $k$-dimensional cycle $Z$ is a formal sum $\sum \alpha_\iota S_\iota$ where $\alpha_\iota>0$ are integers and $\{S_\iota\}$ is a locally finite family of irreducible $k$-dimensional analytic sets; then the analytic set $|Z|:=\bigcup S_\iota$ is called the \textit{support} of $Z$; for details see \cite{T}.}) $Z_\nu$, $\nu=0,1,2,\dots$ in some $\Omega\subset{\C}^m$ (of course, everything can be carried over to manifolds).

\begin{dfn}[Tworzewski \cite{T}]
We say that $Z_\nu$ \textit{converges to $Z_0$ in the sense of Tworzewski}, if\begin{itemize}
\item the supports $|Z_\nu|\stackrel{K}{\longrightarrow} |Z_0|$;
\item for any regular point $a\in\mathrm{Reg}|Z_0|$ and any relatively compact manifold $M$ of complementary dimension, transversal to $|Z_0|$ and $a$ and such that $\overline{M}\cap|Z_0|=\{a\}$, we have for the \textit{total number of intersection} (\footnote{By \cite{TW}, almost all interesections $|Z_\nu|\cap M$ are discrete and so finite. Then the total number of intersection is the formal sum of the intersection points with their respective Draper intersection indices \cite{Dr} taken into account.}) $\mathrm{deg}(Z_\nu\cdot M)=\mathrm{deg}(Z_0\cdot M)$ from some index $\nu_0$ onwards.
\end{itemize}
We will call $M$ \textit{a testing manifold for $Z_0$ at $a$}.
\end{dfn}
\begin{rem}
As noted by Alain Yger \cite{Y}, this convergence is precisely \textit{the weak convergence of the corresponding integration currents} $[Z_\nu]$. See also the general though not very precise discussion in \cite{Ch}.

By \cite{T} Lemma 3.2 it is sufficient to consider testing manifolds at a dense subset of the regular points of $|Z_0|$.

Of course, the definition may be extended to families $\{Z_t\}$ where $t$ belongs to a topological space $T$.
\end{rem}

It will be useful to state clearly the following observation being a mere corollary to the result of \cite{TW}:
\begin{prop}\label{wlasciwosc}
If $X_0,Y_0$ are analytic subsets of an open set $\Omega\subset{\C}^m$ of pure dimensions $p,q$ respectively, and if $X_0\cap Y_0$ has pure dimension $p+q-m$, then for any sequences $X_\nu\stackrel{K}{\longrightarrow} X_0$ and $Y_\nu\stackrel{K}{\longrightarrow} Y_0$ of analytic subsets of $\Omega$ of pure dimension $p$ and $q$ respectively, the intersections $X_\nu\cap Y_\nu$ are proper (i.e. of pure dimension $p+q-m$) for all indices large enough.
\end{prop}
\begin{proof}
By \cite{TW} we know that $X_\nu\cap Y_\nu\stackrel{K}{\longrightarrow} X_0\cap Y_0$. Besides, at any $a\in X_\nu\cap Y_\nu$ we obviously have $\dim_a X_\nu\cap Y_\nu\geq p+q-m$. 

Now fix a point $a\in X_0\cap Y_0$ and choose coordinates in such a way that in a bounded neighbourhood $W=U\times V\subset{\C}^{p+q-m}\times{\C}^{2m-p-q}$ of $a$ the natural projection onto $U$ restricted to the set $Z_0=X_0\cap Y_0$ is a branched covering. We may ask that $(\overline{U}\times\partial V)\cap Z_0=\varnothing$. Write $Z_\nu:=X_\nu\cap Y_\nu\cap W$. Then, by the convergence, for all indices large enough, $(\overline{U}\times\partial V)\cap Z_\nu=\varnothing$, whereas $Z_\nu\neq\varnothing$. 

This means that any such $Z_\nu$ projects properly on $U$. Therefore, if we pick a point $z\in Z_\nu$ and an arbitrarily small polydisc around it, then by the Remmert Proper Map Theorem, $\dim_z Z_\nu\leq p+q-m$. This implies that all the $Z_\nu$'s have pure dimension $p+q-m$. 

Since any subsequence of $X_\nu\cap Y_\nu$ converges to $X_0\cap Y_0$ the proof is accomplished.
\end{proof}

Finally, we briefly recall the notion of \textit{c-holomorphic functions} (cf. \cite{R} and \cite{Wh}) i.e. complex continuous functions that are defined on an analytic set $A$ and holomorphic at its regular points $\mathrm{Reg} A$. We denote by $\mathcal{O}_c(A)$ their ring for a fixed $A$. Their study from the geometric point of view was carried to some extent in \cite{D1}--\cite{D4}. They share many a property of holomorphic functions, though they form a larger class without really useful differential properties. Their main feature is the fact that they are characterized among all the continuous functions $A\to {\C}$ by the analycity of their graphs (see \cite{Wh}). That allows the use of geometric methods. In particular there is an identity principle on irreducible sets (cf. \cite{D2}) and we can consider the \textit{order of vanishing} (see \cite{D1} where it is introduced and studied) at a point $f(a)=0$ (when $f\not\equiv 0$) as 
$$
\mathrm{ord}_a f:=\max\{\eta>0\mid |f(x)|\leq\mathrm{const.}||x||^\eta, \textrm{in a neighbourhood of}\ a\in A\}.
$$

\section{Continuity principle}

\begin{lem}
Let $E\subset{\Rz}^k_t\times{\Rz}^n_x$ be a closed, nonempty set with continuously varying sections $E_t$ over $F:=\pi(E)$ where $\pi(t,x)=t$. Then the function
$$
\delta(t,x):=\mathrm{dist}(x,E_t),\quad (t,x)\in F\times{\Rz}^n
$$
is continuous.
\end{lem}
\begin{proof}
The function $\delta(t,\cdot)$ is $1$-Lipschitz which means that $\lim_{x\to x_0}\delta(t,x)=\delta(t,x_0)$ is uniform with respect to $t$. Therefore, in view of the Iterated Limits Theorem, we need only to check that $t\mapsto \delta(t,x)$ is continuous for all $x$. Indeed, then 
$$
\lim_{(t,x)\to (t_0,x_0)}\delta(t,x)=\lim_{x\to x_0}\delta(t_0,x)=\delta(t_0,x_0).
$$

Fix $(t_0,x_0)$. We know that $E_t\to E_{t_0}$ in the sense of Kuratowski. Then let $d:=d(x_0,E_{t_0})$. In particular, for any $\varepsilon>0$,
$$
\mathbb{B}(x_0,d+\varepsilon)\cap E_{t_0}\neq\varnothing\>\textrm{and}\> \overline{\mathbb{B}}(x_0,d-\varepsilon)\cap E_{t_0}=\varnothing.\leqno{(K)}
$$
Then, the convergence implies (cf. \cite{DD} Lemma 2.1) that for all $t$ sufficiently close to $t_0$, condition $(K)$ holds for $E_t$ instead of $E_{t_0}$. That in turn implies that for all such $t$,
$$
d-\varepsilon<\mathrm{dist}(x_0,E_t)<d+\varepsilon
$$
and the proof is complete.
\end{proof}
\begin{rem}
Of course, the lemma is true for a product of metric spaces. In particular we can replace the parameter space ${\Rz}^k$ by a 1st countable topological space $T$, since for such a $T$ the following general Iterated Limits Theorem holds (\footnote{We do not have a reference for this fact, but the proof is obvious.}): if $f\colon T\times X\to Y$ where $X,Y$ are metric spaces with $Y$ complete, is such that \begin{itemize}
\item $\exists \lim_{t\to t_0} f(t,x)=\varphi(x)$ for any $x\in X$;
\item $\exists \lim_{x\to x_0} f(t,x)=\psi(t)$ uniformly in $t$,
\end{itemize}
then there exists $\lim_{(t,x)\to (t_0,x_0)}f(t,x)=\lim_{x\to x_0} f(t_0,x)=\psi(t_0)$.
\end{rem}

\begin{prop}\label{odl}
Consider a pure $(n+k)$-dimensional analytic set $A\subset {\C}^{p}\times D$ with proper projection $\pi(t,z,w)=(t,z)$ onto the product domain $D=U\times V\subset {\C}^n\times{\C}^k$. Then \begin{enumerate}
\item The sections $A_t$ vary continuously;
\item The function $\delta\colon {\C}^p\times D\ni (t,x)\mapsto \mathrm{dist}(x,A_t)\in{\Rz}$ is  continuous. 
\end{enumerate}
\end{prop}
\begin{proof}
Since $A$ is closed, the sections $A_t$ are upper semi-continuous, by \cite{DD} Proposition 2.7, i.e. for any $t_0$,
$$
\limsup_{t\to t_0} A_t\subset A_{t_0}.
$$
We need to check that $A_{t_0}\subset\liminf_{t\to t_0} A_t$. This amounts to proving that for any $x\in A_{t_0}$ and any $t_\nu\to t_0$ we can find points $x_\nu \in A_{t_\nu}$ converging to $x$. Since $\pi$ is a branched covering on $A$, we see that the fibres $\pi^{-1}(\pi(t_\nu,x))\cap A$ converge to the fibre $\pi^{-1}(\pi(t_0,x))\cap A$ containing $(t_0,x)$ which gives exactly what we need and the proof of (1) is complete.

Now (2) follows from the previous lemma.
\end{proof}
\begin{rem}\label{uwaga}
We stress once again that (2) is a simple consequence of (1).
\end{rem}

\begin{lem}\label{jstcg}
Let $T$ be a locally compact topological space and $X\subset{\C}^m$ a nonempty set. If $f\colon T\times X\to {\C}$ is continuous and we write $f_t(x)=f(t,x)$, then $t\to t_0$ in $T$ implies the convergence of graphs:
$$
\Gamma_{f_t}\stackrel{K}{\longrightarrow}\Gamma_{f_{t_0}}.
$$
\end{lem}
\begin{proof}
Note that the graphs in question are pure $k$-dimensional sets (cf. \cite{D1}). 
In view of Remark \ref{zb} we need only to check that for any $t_\nu\to t_0$, $f_{t_\nu}\to f_{t_0}$ locally uniformly on $X$. Take a compact set $K\subset X$. Then $K'=\{t_0\}\times K$ is compact and for a fixed $\varepsilon$ and any $x\in K$ we find neighbourhoods $U_x\times \mathbb{B}(x,r_x)$ of $(t_0,x)$ at points $(t,y)$ of which 
$$
|f(t,y)-f(t_0,x)|<\varepsilon.
$$
By compacity we choose a finite covering $K'\subset\bigcup_{i=1}^pU_i\times\mathbb{B}(x_i,r_i)$ and put $U:=\bigcap_{i=1}^p U_i$. then for any $(t,x)\in U\times K$ we have $(t,x)\in U_i\times\mathbb{B}(x_i,r_i)$ for some $i$ and so
$$
|f(t,x)-f(t_0,x)|\leq \varepsilon.
$$
This ends the proof.
\end{proof}

\begin{prop}
Let $T$ be a locally compact, connected topological space, $A$ a pure $k$-dimensional analytic subset of some open set $\Omega\subset{\C}^m$ and $f\colon T\times A\to {\C}$ a continuous function such that for each $t\in T$, $f_t(x):=f(x,t)$ is c-holomorphic on $A$. Then $t\to t_0$ in $T$ implies
$$
\Gamma_{f_t}\stackrel{T}{\longrightarrow} \Gamma_{f_{t_0}}.
$$
\end{prop}
\begin{proof}
By Lemma \ref{jstcg} we have $$\Gamma_{f_t}\stackrel{K}{\longrightarrow}\Gamma_{f_{t_0}}.$$ This means that on $\mathrm{Reg} A$, for any $t_\nu\to t_0$, we have a sequence of holomorphic functions converging locally uniformly.

Now, observe that for any $g\in\mathcal{O}_c(A)$, $\Gamma_{g|_\mathrm{Reg}A}\subset\mathrm{Reg}\Gamma_g$ is dense. For a testing $M$ at $a\in \Gamma_{f_{t_0}|_\mathrm{Reg}A}$ we have the equality $M\cap T_a\Gamma_{f_{t_0}}=\{0\}$ where $T_a\Gamma_{f_{t_0}}$ denotes the tangent space at $a$, and so $\mathrm{deg}(M\cdot\Gamma_{f_{t_0}})=1$. But since in the holomorphic case, the local  uniform convergence is a convergence with the tangents, we easily conclude that for sufficiently large indices $\nu$, $M$ is transversal to the manifold (near $a$) $\Gamma_{f_t}$ and so $\mathrm{deg}(M\cdot \Gamma_{f_t})=1$, too (there are no multiplicities attached to the graphs). To be somewhat more precise, if $a=(a',f_{t_0}(a'))$, then $$T_{(a',f_{t_\nu}(a'))}\Gamma_{f_{t_\nu}}\stackrel{K}{\longrightarrow}T_{(a',f_{t_0}(a'))}\Gamma_{f_{t_0}}$$
and we apply \cite{TW} to conclude that $M$ intersects $\Gamma_{f_t}$ transversally.
\end{proof}

Recall (cf. \cite{D1}-- \cite{D3}) that if $f\in\mathcal{O}_c(A)$ does not vanish identically on any irreducible component of $A$, where $A$ is a pure $k$-dimensional analytic subset of a domain $D\subset{\C}^m$, then we define the \textit{cycle of zeroes} as the Draper proper intersection  cycle (\cite{Dr})
$$
Z_f:=\Gamma_f\cdot (D\times\{0\}).
$$
In the same way we may define the \textit{fibre cycle}, namely
$$
[f^{-1}(f(a))]:=\Gamma_f\cdot(D\times\{f(a)\})
$$
and consider this as a cycle in $D$.

Now we can state the following Hurwitz-type theorem:
\begin{thm}\label{zera}
Let $T$ be a connected topological space, $A$ a pure $k$-dimensional analytic subset of some domain $D\subset{\C}^m$, $f\colon T\times A\to {\C}$ a continuous function such that for each $t\in T$, $f_t(x):=f(x,t)$ is c-holomorphic on $A$. Then if $f_{t_0}\not\equiv 0$ on any irreducible component of $A$ and $f_{t_0}^{-1}(0)\neq\varnothing$, we have 
$$
Z_{f_t}\stackrel{T}{\longrightarrow} Z_{f_{t_0}},\> t\to t_0.
$$
\end{thm}
\begin{proof}
By the previous Proposition we have 
$$
\Gamma_{f_t}\stackrel{T}{\longrightarrow} \Gamma_{f_{t_0}}.
$$
Of course, $f_{t_0}^{-1}(0)$ is a hypersurface (cf. the identity principle from \cite{D2}) which means that the intersection $\Gamma_{f_{t_0}}\cap (D\times\{0\})$ is proper (i.e of the minimal dimension possible: $k-1$). By \cite{T} Lemma 3.5 we conclude that for any sequence $t_\nu\to t_0$, 
$$
\Gamma_{f_{t_\nu}}\cdot(D\times\{0\})\stackrel{T}{\longrightarrow} \Gamma_{f_{t_0}}\cdot(D\times\{0\}).
$$
This ends the proof.
\end{proof}

\begin{cor}
Let $g\in\mathcal{O}_c(A)$, $g\neq\mathrm{const.}$ on any irreducible component of $A\subset D$, where $A$ is pure $k$-dimensional. Then for any $t_0\in A$,
$$
[g^{-1}(t)]\stackrel{T}{\longrightarrow}[g^{-1}(t_0)],\> t\to t_0.
$$
\end{cor}
\begin{proof}
Let $f\colon A\times {\C}\ni (x,t)\mapsto g(x)-t\in{\C}$. By \cite{D2}, we conclude that all the nonempty fibres of $g$ have pure dimension $k-1$. Then $f$ satisfies the assumptions of the preceding Theorem and 
\begin{align*}
Z_{f_t}&=\Gamma_{f_t}\cdot (D\times\{0\})=\\
&=\Gamma_{g}\cdot (D\times\{t\})=\\
&=[g^{-1}(t)],
\end{align*}
since $\Phi(x,s)=(x,s+t)$ is an automorphism of $D\times {\C}$ sending $\Gamma_{f_t}$ to $\Gamma_g$ and $D\times\{0\}$ to $D\times\{t\}$. This ends the proof.
\end{proof}
Before the next corollary recall that for any positive cycle $Z=\sum\alpha_\iota S_\iota$ we define its \textit{local degree} at $a\in|Z|$ as $\mathrm{deg}_a Z:=\sum\alpha_\iota\mathrm{deg}_a S_\iota$, where $\mathrm{deg}_a S_\iota$ is the usual local degree (Lelong number) with the convention that $\mathrm{deg}_a S_\iota=0$ if $a\notin S_\iota$.
\begin{cor}
Under the assumptions of the preceding Theorem suppose in addition that $f_t(a)=0$ for all $t\in T$ and some fixed $a\in A$. Then for all $t$ close enough to $t_0$,
$$
\mathrm{deg}_a Z_{f_t}\leq\mathrm{deg}_a Z_{f_{t_0}},
$$
for the local degrees at $a$.
\end{cor}
\begin{proof}\label{rzqd}
Take any affine subspace $L$ through $a$, of dimension $m-k+1$ and such that $$L\cdot Z_{f_{t_0}}=\mathrm{deg}_a Z_{f_{t_0}}\cdot \{a\}.$$
Then by Theorem \ref{zera} together with \cite{T} Lemma 3.5,
$$
L\cdot Z_{f_{t}}\stackrel{T}{\longrightarrow}L\cdot Z_{f_{t_0}}
$$
which ends the proof, since $$L\cdot Z_{f_t}=\sum_{b\in L\cap f_t^{-1}(0)}i(L\cdot Z_{f_t},b)\{b\}$$ and for each Draper intersection index (multiplicity) $i(L\cdot Z_{f_t},b)$ we have $$i(L\cdot Z_{f_t},b)\geq \mathrm{deg}_b Z_{f_t},$$ for $\mathrm{deg}_b L=1$. Therefore, we obtain by the convergence, for all $t$ sufficiently close to $t_0$,
\begin{align*}
\mathrm{deg}_a Z_{f_{t_0}}&=\mathrm{deg}(L\cdot Z_{f_{t_0}})=\\
&=\mathrm{deg}(L \cdot Z_{f_t})=\\
&=\sum_{b\in L\cap f_t^{-1}(0)}i(L\cdot Z_{f_t},b)\{b\}\geq\\
&\geq i(L\cdot Z_{f_t},a)\{a\}\geq \mathrm{deg}_ a Z_{f_t},
\end{align*}
as $a\in L\cap f_t^{-1}(0)$ (for all $t$).
\end{proof}

\section{On the \L ojasiewicz inequality and the total degree}

We recall one result from \cite{D3} which is the basis which we shall work upon.
\begin{thm}[\cite{D3} Theorem 2.3]\label{ordLoj}
Let $f\colon\Omega\to{\C}$ be holomorphic in a (connected) neighbourhood $\Omega$ of $0\in {\C}^m$. If $f$ is non-constant and $f(0)=0$ then there is a neighbourhood $U$ of zero such that the following \L ojasiewicz inequality holds:
$$
|f(x)|\geq\mathrm{const.}\mathrm{dist}(x,f^{-1}(0))^{\mathrm{ord}_0 f},\quad x\in U
$$
where $\mathrm{ord}_0 f$ denotes the order of vanishing of $f$ at zero. Moreover, this is the best exponent possible.
\end{thm}

As before we consider the intersection cycle of zeroes 
$
Z_f=\Gamma_f\cdot(\Omega\times\{0\}).
$
\begin{prop}[\cite{D3} Proposition 2.1]\label{ord}
In the setting introduced above, $\mathrm{deg}_0 Z_f=\mathrm{ord}_0 f$.
\end{prop}

We easily generalize these results to c-holomorphic functions, although only in a weak sense (compare the following theorem with the results of \cite{D4}). Consider a pure $k$-dimensional ($k\geq 2$) analytic subset $A$ of a neighbourhood $\Omega$ of $0\in {\C}^m$ with $0\in A$. Assume that $f\in\mathcal{O}_c(A)$ satisfies $f(0)=0$ and does not vanish identically on any irreducible component of $A$ containing zero.

\begin{thm}\label{cLoj}
In the c-holomorphic setting introduced above, there is a neighbourhood $W$ of zero such that
$$
|f(z)|\geq\mathrm{const.}\mathrm{dist}(z,f^{-1}(0))^{\mathrm{deg}_0 Z_f\cdot\mathrm{deg}_0 f^{-1}(0)},\quad z\in W\cap A.
$$
\end{thm}
\begin{proof}
Write ${\C}^m={\C}^{k-1}\times{\C}^{m-k+1}$ with coordinates $(x,y)$.

We may assume that the coordinates are chosen in such a way that the projection $\pi(x,y)=x$ onto the first $k-1$ coordinates is proper on $Z:=f^{-1}(0)\cap (U\times V)$ with covering number equal to the local degree $\mathrm{deg}_0 f^{-1}(0)=:d$. Here $U\times V$ is a neighbourhood of the origin satisfying $(\{0\}\times \overline{V})\cap f^{-1}(0)=\{0\}$.

Applying Proposition 2.2 from \cite{CgT} we find a holomorphic mapping $F\colon U\times{\C}^{m-k+1}\to {\C}^p$ such that $F^{-1}(0)=f^{-1}(0)\cap (U\times V)$ and 
$$
||F(x,y)||\geq\mathrm{dist}((x,y), Z)^d,\quad (x,y)\in U\times{\C}^{m-k+1}.\leqno{(*)}
$$

If we write $F=(F_1,\dots, F_p)$ we observe that $F_j^{-1}(0)\cap A\supset f^{-1}(0)\cap (U\times V)$ for all $j$. The intersection of the graph $\Gamma_f$ with $\Omega\times\{0\}$ being proper, we can now apply the c-holomorphic Nullstellensatz from \cite{D3}. In other words, we find a neighbourhood $W\subset U\times V$ of zero and $p$ c-holomorphic functions $h_{j}$ on $W\cap A$ for which
$$
F_j^\delta=h_{j} f\ \textrm{on}\ A\cap W,\ j=1,\dots, p\leqno{(**)}
$$
with $\delta=\mathrm{deg}_0 Z_f$.

Combining $(*)$ and $(**)$ we eventually obtain the inequality looked for.
\end{proof}

\begin{prop}
Under the assumptions of the previous theorem, 
$$
\mathrm{deg}_0 Z_f\cdot\mathrm{deg}_0 f^{-1}(0)\geq \mathrm{ord}_0 f.
$$
\end{prop}
\begin{proof}
This follows from Lemma 4.8 in \cite{D1}.
\end{proof}

Using Corollary \ref{rzqd} and Proposition \ref{ord} we easily obtain
\begin{lem}
If $f=f(t,x)\in\mathcal{O}_{k+m}$ is such that $f_t(0):=f(t,0)=0$ for all $t$ small enough and $f_0=f(0,\cdot)$ is non-constant, then 
$$
\mathrm{ord}_0 f_t\leq \mathrm{ord}_0 f_0
$$
for all $t$ sufficiently close to zero.
\end{lem}
\begin{ex}
The inequality may be strict as we easily see by taking $f(t,x)=tx+x^2$; then for $t\neq 0$, $\mathrm{ord}_0 f_t=1<\mathrm{ord}_0 f_0=2=\mathrm{ord}_0 f$. But of course there is no direct relation with $\mathrm{ord}_0 f$, it suffices to take $f(t,x)=tx+x^3$ in order to have $\mathrm{ord}_0 f_t=1<\mathrm{ord}_0 f=2<\mathrm{ord}_0 f$. 
\end{ex}

The proof of Theorem \ref{ordLoj} suggests the following result.
\begin{prop}\label{kluczowe}
Let $V\times W\subset\subset{\C}^{m-1}\times{\C}$ be a bounded, connected neighbourhood of zero (a polydisc) and let $P\in\mathcal{O}(V)[t]$ be unitary and  such that $P^{-1}(0)\subset (V\times W)$ projects properly onto $V$. Then in $V\times W$ there is 
$$
|P(x,t)|\geq \mathrm{dist}((x,t),P^{-1}(0))^\delta
$$
with $\delta=\mathrm{deg}((\{0\}^{m-1}\times W)\cdot Z_P)$.
\end{prop}
\begin{proof}
Recall from \cite{D3} that $Z_P=\sum\alpha_j S_j$ where $S_j$ are the irreducible components of $P^{-1}(0)$ and $\alpha_j=\min\{\mathrm{ord}_z P\mid z\in \mathrm{Reg} S_j\}$ is the generic order of vanishing of $P$ along $S_j$. Note that each $S_j$ projects onto the whole of $V$. 

Now, since the intersections $(\{x\}\times W)\cap P^{-1}(0)$ are proper, by \cite{T} (see also \cite{Ch}) we conclude that for any $x_\nu\to 0$ we have
$$
(\{x_\nu\}^{m-1}\times W)\cdot Z_P\stackrel{T}{\longrightarrow}(\{0\}^{m-1}\times W)\cdot Z_P
$$
and so $\mathrm{deg}((\{0\}^{m-1}\times W)\cdot Z_P)=\delta$ for sufficiently large $\nu$. 

Observe that for the generic $x\in V$ we have the following situation: $\{x\}\times W$ intersects $P^{-1}(0)$ transversally at $d$ regular points $b^{(i)}=(x,t^{(i)})$, where $d$ is the multiplicity of the branched covering $P^{-1}(0)\to V$, each of these points belongs to exactly one $S_j$, all the $S_j$'s appear in this assignment, and $\mathrm{ord}_{b^{(i)}} P=\alpha_j$ for the unique $j$ such that $b^{(i)}\in S_j$. Therefore, we may write
$$
\delta=\sum_{b\in (\{x\}\times W)\cap P^{-1}(0)}\mathrm{ord}_b P.
$$

On the other hand, for any such point $x$ we have
$$
P(x,t)=\prod_{i=1}^d(t-t^{(i)})^{n_i}
$$
with $n_i$ independent of the point chosen. We observe that $n_i=\mathrm{ord}_{b^{(i)}} P$. Indeed, if we write $\{x\}\times W$ as the zero-set of an affine mapping $\ell=(\ell_1,\dots, \ell_{m-1})$ restricted to $V\times W$, then the transversality of the intersection $(\{x\}\times W)\cap P^{-1}(0)$ implies by the Tsikh-Yuzhakov result (see \cite{Ch}) that the multiplicity $m_{b^{(i)}}(P,\ell)$ at each point $b^{(i)}$ of the proper mapping germ $(P,\ell)$ is equal to the product of the orders of $P$ and the $\ell_j$'s, i.e. to $\mathrm{ord}_{b^{(i)}} P$. On the other hand, by \cite{Ch} p. 107-108 we easily see that $$m_{b^{(i)}}(P,\ell)=\mathrm{ord}_{t^{(i)}} P|_{\{x\}\times W}=n_i.$$

Therefore, $\delta=\sum_{i=1}^d n_i$. This allows us to write, for the generic $x\in V$, the following inequalities:
\begin{align*}
|P(x,t)|&=\prod_{i=1}^d|t-t^{(i)}|^{n_i}=\\
&=\prod_{i=1}^d||(x,t)-(x,t^{(i)})||^{n_i}\geq \\
&\geq \mathrm{dist}((x,t),P^{-1}(0))^{\sum_{i=1}^d n_i}.
\end{align*}
Extending this by continuity to the whole of $V\times W$ ends the proof.
\end{proof}
\begin{rem}
The proof above is in fact an extrapolation of the proof
of Theorem \ref{ordLoj}, where we use the Weierstrass Preparation in a neighbourhood of zero such that $(\{0\}\times W)\cap f^{-1}(0)=\{0\}$ and $\mathrm{ord}_0 f=\mathrm{ord}_0 P$. 
\end{rem}
\begin{cor}
If $f\colon V\times W\to {\C}$ is a holomorphic function such that $f^{-1}(0)$ projects properly onto $V$, then for some possibly smaller neighbourhood $U\subset V\times W$ of zero, $f$ satisfies the \L ojasiewicz inequality in $U$ with exponent $\mathrm{deg}((\{0\}\times W)\cdot Z_f)$.
\end{cor}
\begin{proof}
In $V\times W$ we can apply the Weierstrass Preparation Theorem and write $f=hP$ with a holomorphic function $h$ such that $h^{-1}(0)=\varnothing$. Shrinking the neighbourhood (actually, we need only to shrink $V$ if any), we may assume that $\inf|h|>0$. Then $Z_f=Z_P$, since $\mathrm{ord}_b f=\mathrm{ord}_b P$. The preceding Proposition gives the result.
\end{proof}

\section{The \L ojasiewicz inequality with parameter}

Eventually, we are ready to prove the main result.
\begin{thm}
Assume that $f\colon T\times\Omega\to{\C}$ is a continuous function where $T$ is a locally compact, connected topological space,  $\Omega\subset{\C}^m$ is a domain, and for all $t\in T$, $f_t\in\mathcal{O}(\Omega)$ does not vanish identically. Assume moreover that $0\in\Omega$ and $f_{t}(0)=0$ for any $t$. Then there is a neighbourhood $U\subset\Omega$ of zero such that, for all $t$ close enough to $t_0$, $$|f_t(x)|\geq c(t)\mathrm{dist}(x,f_t^{-1}(0))^{\alpha},\quad x\in U$$ where $c(t)>0$ is a constant depending on the parameter, but the exponent $$\alpha=\mathrm{ord}_0 f_{t_0}$$ is uniform.
\end{thm}
\begin{proof}
By Theorem \ref{zera} we know in particular that $f_t^{-1}(0)\stackrel{K}{\longrightarrow} f_{t_0}^{-1}(0)$. Of course these sets are hypersurfaces. The type of convergence implies that we can choose coordinates in ${\C}^m$ in such a way that for some neighbourhood $V\times W\subset{\C}^{m-1}\times{\C}$ of zero, $V$ connected and $W$ a disc, we have 
$$
f_t^{-1}(0)\cap (V\times\partial W)=\varnothing
$$
for all $t$ close enough to $t_0$. This means that the zero-sets intersected with $V\times W$ project properly onto $V$. Moreover, we may assume that $$(\{0\}^{m-1}\times W)\cdot Z_{f_{t_0}}=\mathrm{ord}_0 f_{t_0}\{0\}.$$ 

In the situation considered, the proof of Proposition \ref{kluczowe} shows that the \L ojasiewicz inequality for $f_{t_0}$ is satisfied in $V\times W$ with the exponent $d_t=\mathrm{deg}((\{0\}\times W)\cdot Z_{f_t})$:
$$
|f_t(x)|\geq c(t)\mathrm{dist}(x,f_t^{-1}(0))^{d_t},\quad x\in V\times W\leqno{(*)}
$$
where $c(t)>0$ is a constant.

But then, 
for $t$ close enough to $t_0$, the numbers $d_t$ fortunately coincide with $(\{0\}\times W)\cdot Z_{f_{t_0}}=\mathrm{ord}_0 f_{t_0}$ by the convergence (Theorem \ref{zera}). 


This ends the proof.
\end{proof}

It seems hard to obtain a satisfactory c-holomorphic counter-part to this Theorem due to the use of the Nullstellensatz with parameter. The best we were able to obtain is the following Theorem.

\begin{thm}
Assume that $f\colon T\times A\to{\C}$ is a continuous function where $T$ is a locally compact, connected topological space, $A$ is a pure $k$-dimensional analyti subset of an open set $\Omega\subset{\C}^m$, $0\in A$, and for all $t\in T$, $f_t\in\mathcal{O}_c(A)$ does not vanish identically on any irreducible compponent of $A$ through zero. Assume moreover that $f_{t}(0)=0$ for any $t$. Then there is a neighbourhood $U\subset\Omega$ of zero such that, for all $t$ close enough to $t_0$, $$|f_t(x)|\geq c(t)\mathrm{dist}(x,f_t^{-1}(0))^{\alpha},\quad x\in A\cap U$$ where $c(t)>0$ is a constant depending on the parameter, but the exponent $$\alpha=(\deg_0 Z_{f_{t_0}})^2 $$ is uniform.
\end{thm}
\begin{proof}
We give the proof in several steps. 

\textbf{Step 1.} Choose coordinates in ${\C}^m$ in such a way that $A$ projects properly onto the first $k$ coordinates and, moreover, 
$$
i((\{0\}^{k-1}\times{\C}^{m-k+1})\cdot Z_{f_{t_0}};0)=\mathrm{deg}_0 Z_{f_{t_0}}.
$$
Let $\ell\colon {\C}^m\to {\C}^{k-1}$ be the linear epimorphism whose kernel is exactly $\{0\}^{k-1}\times{\C}^{m-k+1}$. Write $$\varphi_t\colon A\ni x\mapsto (f_t(x),\ell(x))\in{\C}\times{\C}^{k-1}$$ for $t\in T$. Fix a polydisc $V\times W\subset{\C}^{k-1}\times{\C}^{m+k-1}$ centred at zero such that $$(\{0\}^{k-1}\times\overline{W})\cap f_{t_0}^{-1}(0)=\{0\}.$$
In particular we may assume that $f_{t_0}^{-1}(0)$ projects properly onto $V$.

\textbf{Step 2.} The latter intersection corresponds to $(\overline{V\times W}\times \{0\}^k)\cap\Gamma_{\varphi_{t_0}}$
which means that there is a polydisc $P\subset{\C}^k$ such that the pure $k$-dimensional analytic set $(V\times W\times P)\cap \Gamma_{\varphi_{t_0}}$ projects properly \textit{onto} $P$ along $V\times W$. In other words, $\varphi_{t_0}|_{(V\times W)\cap A}$ is proper with image $P$.

As in Lemma \ref{jstcg}, the continuity of $$\Phi\colon T\times A\ni (t,x)\mapsto \varphi_t(x)\in{\C}^k$$
implies the Kuratowski convergence of the graphs $\Gamma_{\varphi_t}\stackrel{K}{\longrightarrow}\Gamma_{\varphi_{t_0}}$ as $t\to t_0$. Therefore, by the same argument as in Proposition \ref{wlasciwosc}, we conclude that for all $t$ close enough to $t_0$, the restrictions of the natural projection
$$\pi_t\colon (V\times W\times P)\cap \Gamma_{\varphi_{t}}\to P$$ 
are branched coverings. In particular, all these $\varphi_t$ have the same image $P$. Let $q_t$ denote the multiplicity of the branched covering $\varphi_t|_{A\cap (V\times W)}$. 

\textbf{Step 3.} By the choice of $V\times W$ and Theorem \ref{zera}, we know (cf. the proof of the previous Theorem) that for all $t$ close enough to $t_0$, the zero-sets $f_t^{-1}(0)\cap (V\times W)$ project properly onto $V$. Let $d_t$ denote the multiplicity of such a branched covering.

Since by Theorem \ref{zera} we know that the cycles of zeros of the restrictions $f_t|_{A\cap (V\times W)}$ converge with $t\to t_0$ in the sense of Tworzewski, we easily conclude from \cite{T} Lemma 3.5  and \cite{TW} that 
$$
d_{t_0}\leq d_t\leq \deg((\{0\}\times W)\cdot Z_{f_t})=\deg((\{0\}\times W)\cdot Z_{f_{t_0}})=\deg_0 Z_{f_{t_0}}.\leqno{(\star)}
$$
On the other hand, we observe that $q_t=\deg((\{0\}\times W)\cdot Z_{f_t})$ and so 
$$
q_t\leq \deg_0 Z_{f_{t_0}}.\leqno{(\star\star)}
$$
Indeed, it is easy to see that $q_t$ is in fact the multiplicity of the projection
$$
\pi\colon {\C}^{k-1}\times{\C}^{m-k+1}\times{\C}\ni(u,v,w)\mapsto (w,u)\in{\C}\times {\C}^{k-1}
$$
over $P$ when restricted to $\Gamma_t:=\Gamma_{f_t}\cap (V\times W\times {\C})$. This, in turn, by the classical Stoll Formula, is the total degree of the intersection cycle $\pi^{-1}(0)\cdot \Gamma_t$ In other words, 
$$
q_t=\deg((V\times W\times\{0\})\cdot \Gamma_{t}).
$$
However, in view of \cite{TW2} Theorem 2.2, we can write
\begin{align*}
(V\times W\times\{0\})\cdot \Gamma_{t}&=(\{0\}\times W)\cdot_{V\times W\times\{0\}}((V\times W\times\{0\})\cdot \Gamma_t)=\\
&=(\{0\}\times W)\cdot Z_{f_t|_{A\cap (V\times W)}}=\\
&=(\{0\}\times W)\cdot Z_{f_t}.
\end{align*}

\textbf{Step 4.} As in the proof of Thoerem \ref{cLoj}, by \cite{CgT} Proposition 2.2 we know that for each $t$ close to $t_0$ there are $p_t=d_t(m-k)+1$ holomorphic functions $F_{t,j}\colon V\times{\C}^{m-k+1}\to{\C}$ whose common zeroes form coincide with the set $f_t^{-1}(0)\cap (V\times W)$ and for which
$$
||(F_{t,1},\dots, F_{t,p_t})(x)||\geq\mathrm{dist}(x, f_t^{-1}(0)\cap (V\times W))^{d_t}
$$
for all $x\in V\times W$. 

Now, we can apply Lemma 3.1 from \cite{D3} (compare \cite{PT}) in order to get \textit{on the whole} of $A\cap (V\times W)$,
$$
F_{t,j}^{q_t}=h_{t,j} f_t, \quad j=1,\dots, p_t,
$$
with some functions $h_{t,j}\in\mathcal{O}_c(A\cap (V\times W))$.

This leads to the inequalities
$$
|f_t(x)|\geq c(t)\mathrm{dist}(x, f_t^{-1}(0))^{p_tq_t},\quad x\in A\cap (V\times W)\leqno{(\#)}
$$
for all $t$ close to $t_0$ and some constants $c(t)>0$. 

\textbf{Step 5.} Thanks to the continuity of the zero-sets (cf. Theorem \ref{zera}), Proposition \ref{odl} (cf. Remark \ref{uwaga}) allows us to choose an arbitrarily small neighbourhood $T_0$ of $t_0$ and a neighbourhood $U\subset V\times W$ of zero such that for all $t\in T_0$ and all $x\in U$, we have 
$$
\mathrm{dist}(x,f_t^{-1}(0))<1.
$$
Therefore, we may increase \textit{ad libitum} the exponent in $(\#)$, provided $x\in A\cap U$. The estimates $(\star)$ and $(\star\star)$ end the proof.
\end{proof}

\section{Acknowledgements}

During the preparation of this paper the author was partially supported by Polish Ministry of Science and Higher Education grant  IP2011 009571. 

This article was revised and finished during the author's stay at the University Lille 1 where he found excellent working conditions.


\begin{thebibliography}{TW2}
\bibitem[Ch]{Ch} E. M. Chirka, {\it Complex Analytic Sets}, Kluwer Acad. Publ. 1989;

\bibitem[B]{B} M. Bilski, {\it Approximation of analytic sets with proper projection by algebraic sets}, Constr. Approx. 35 (2012), no. 3, 273-291;

\bibitem[CgT]{CgT} E. Cygan, P. Tworzewski, {\it Proper intersection multiplicity and regular
separation of analytic sets}, Ann. Polon. Math. LIX.3 (1994),  293-298;

\bibitem[DD]{DD} Z. Denkowska, M. P. Denkowski, {\it Kuratowski convergence and connected components}, J. Math. Anal. Appl. (2012);

\bibitem[D1]{D1} M. P. Denkowski, {\it The \L ojasiewicz exponent of c-holomorphic mappings}, Ann. Polon. Math. 87 no. 1 (2005), 63-81;

\bibitem[D2]{D2} M. P. Denkowski, {\it A note on the Nullstellensatz for c-holomorphic functions}, Ann. Polon. Math. 90 no. 3 (2007),  219-228;

\bibitem[D3]{D3} M. P. Denkowski, {\it On the order of holomorphic and c-holomorphic functions}, Univ. Iagell. Acta Math. XLV (2007), pp. 47-61;

\bibitem[D4]{D4} M. P. Denkowski, {\it Regular separation with parameter of complex analytic sets}, Kodai Math. Journ. 30 (2007),  429-437;

\bibitem[D5]{D5} M. P. Denkowski, {\it The complex gradient inequality with parameter}, (2014) submitted;

\bibitem[DP]{DP} M. P. Denkowski, R. Pierzcha\l a, {\it On the Kuratowski convergence of analytic
sets}, Ann. Polon. Math. 93.2 (2008), pp. 101-112;

\bibitem[Dr]{Dr} R. N. Draper, {\it Intersection theory in analytic geometry}, Math. Ann. 180 (1969), 175-204; 

\bibitem[\L]{L} S. \L ojasiewicz, {\it Ensembles semi-analytiques}, preprint IHES 1965;

\bibitem[\L W]{LW} S. \L ojasiewicz, K. Wachta, {\it S\'eparation r\'eguli\`ere avec param\`etre pour les sous-analytiques}, Bull. Pol. Ac. Sci. XXX no 7-8 (1982);


\bibitem[PT]{PT} A. P\l oski, P. Tworzewski, {\it Effective Nullstellensatz on analytic and algebraic varieties}, Bull. Polish Acad. Sci. Math. 46 (1998), 31-38;

\bibitem[R]{R} R. Remmert, {\it Projektionen analytischer Mengen}, Math. Ann. 130 (1956), 410-441; 

\bibitem[T]{T} P. Tworzewski, {\it Intersection theory in complex analytic geometry}, Ann. Polon. Math. LXII.2 (1995), pp. 177-191;

\bibitem[TW1]{TW} P. Tworzewski, T. Winiarski, {\it Continuity of intersection of analytic sets}, Ann. Polon. Math. 42 (1983), 387-393;

\bibitem[TW2]{TW2} P. Tworzewski, T. Winiarski, {\it Cycles of zeroes of holomoprhic mappings}, Bull. Polish Acad. Sci. Math. 37 (1986), 95-101;


\bibitem[Wh]{Wh} H. Whitney, {\it Complex Analytic Varieties}, Addison-Wesley Publ. Co. 1972.

\bibitem[Y]{Y} A. Yger, {\it Quatre expos\'es sur la th\'eorie de l'intersection} {\tt http://www.math.u-bordeaux1.fr/$\widetilde{\hbox{}}$ayger/th$\hbox{}_-$inter98.pdf}
\end{thebibliography}
\end{document}